\documentclass[12pt]{article}

\usepackage[top=2cm, bottom=2cm, left=2cm, right=2cm]{geometry}

\usepackage{amsmath, amsthm, bm, mathrsfs, mathtools, amssymb, amsfonts}
\usepackage{graphicx}
\usepackage{amsfonts}
\usepackage[colorlinks=true, allcolors=blue]{hyperref}
\usepackage{multicol}

\newtheorem{theorem}{Theorem}

\newtheorem{lemma}{Lemma}
\newtheorem{remark}{Remark}

\begin{document}

\title{Strong law of large numbers for a function of the local times of a transient random walk on groups}

\author{Yinshan Chang\thanks{Address: College of Mathematics, Sichuan University, Chengdu 610065, China; Email: ychang@scu.edu.cn.}, Qinwei Chen\thanks{Address: College of Mathematics, Sichuan University, Chengdu 610065, China; Email: 2023222010057@stu.scu.edu.cn.}, Qian Meng\thanks{Address: Department of Statistics, University of Washington, Seatle 98195, USA; Email: qmeng4@uw.edu.}, Xue Peng\thanks{Address: College of Mathematics, Sichuan University, Chengdu 610065, China; Email: pengxuemath@scu.edu.cn; Supported by National Natural Science Foundation of China \#12001389.}}

\date{}

\maketitle

\begin{abstract}
This paper presents the strong law of large numbers for a function of the local times of a transient random walk on groups, extending the research of Asymont and Korshunov \cite{AsymontKorshunovMR4166201} for random walks on the integer lattice $\mathbb{Z}^d$. Under some weaker conditions, we prove that a certain function of the local times converges almost surely and in $L^1$ and $L^2$. The proof is mainly based on the subadditive ergodic theorem.
\end{abstract}

\section{Introduction}

Suppose that $(S_n)_{n\geq 0}$ is a transient random walk on a countable group $G$ starting from the identity element $e$, i.e. $P(S_0=e)=1$. For $n\geq 1$, let $\xi_n=S_{n-1}^{-1}S_n$ be the $n$-th increment of the random walk.
Then, $S_n=\xi_1\xi_2\cdots\xi_n$. And we assume that $(\xi_n)_{n\geq 1}$ is a sequence of independent and identically distributed (i.i.d.) random variables taking values in $G$.

The total number of visits to an element $x\in G$ up to time $n$ is called the local time $\ell(n,x)$ of $x$, i.e.
\[\ell(n,x)=\sum_{i=0}^{n}\mathbb{I}_{(S_i=x)}.\]
For a function $f:\mathbb{Z}_{+}\to\mathbb{R}$ with $f(0)=0$, define
\[G_n(f)=\sum_{x\in G}f(\ell(n-1,x)).\]
More generally, for $0\leq m<n,$ define
\[G_{m,n}(f)=\sum_{x\in G}f(\ell(n-1,x)-\ell(m-1,x)),\]
where we define $\ell(-1,x)=0$. Note that $G_{n}(f)=G_{0,n}(f)$.

In this paper we are interested in the law of large numbers for $G_{n}(f)$ as $n\to\infty$.

Let us present some previous results. If $f(i)=\mathbb{I}_{(i\geq 1)}$, $G_{n}(f)$ represents the range $R_n$ of the random walk. Dvoretzky and Erd{\"o}s \cite{DvoretzkyErdosMR47272} showed that $\lim_{n\to\infty}G_n(f)/n=\gamma$ with probability $1$ for random walks in $d$-dimensional Euclidean space with $d\geq 3$, where $\gamma$ is the escaping probablity $P(S_n\neq 0,\forall n\geq 1)$. By different methods, Kesten, Spitzer and Whitman \cite[Section~4]{SpitzerMR388547} have obtained the same result with generalizations to the number of points swept out by a finite set. For other interesting results on the range of random walks, we refer to \cite{AsselahSchapiraSousiMR3945751,OkamuraMR4289898,MrazovicSandricSebekMR4643004,DemboOkadaMR4791423,GilchMR4718437} and the references there. If $f(i)=\mathbb{I}_{(i=j)}$, $G_{n}(f)$ represents the $j$-multiple range, i.e. the range of the points visited exactly $j$ times. In \cite{PittMR386021}, Pitt showed that $\lim_{n\to\infty}G_{n}(f)/n=\gamma^2(1-\gamma)^{j-1}$ with probability $1$ for random walks on a countable Abelian group, where $\gamma$ is the probability that the random walk does not return to the initial position. In \cite{DerriennicMR588163}, Derriennic get the same conclusion for random walks on countable groups. For other interesting results on the multiple range of random walks, we refer to \cite{HamanaMR1434120,HamanaMR1608981} and the references there. If $f(i)=i^{\alpha}$, then $G_{n}(f)$ represents the number of $\alpha$-fold self-intersections of the random walk up to time $n$. In \cite{BeckerKonigMR2501325}, Becker and K{\"o}nig showed that $\lim_{n\to\infty}G_{n}(f)/n=\sum_{j=1}^{\infty}j^{\alpha}\gamma^2(1-\gamma)^{j-1}$ with probability 1 for random walks on $\mathbb{Z}^d$. For general $f$ satisfying  $\sum_{j=1}^\infty|f(j)|^2j(1-\gamma)^j<\infty$, Asymont and Korshunov \cite[Theorem~1]{AsymontKorshunovMR4166201} showed that $\lim_{n\to\infty}G_n(f)/n=\gamma^2\sum_{j=1}^\infty f(j)(1-\gamma)^{j-1}$ in $L^2$ and almost surely for transient random walks on $\mathbb{Z}^d$.

Our main results are the following theorems.
\begin{theorem}\label{thm: main}
For a transient random walk $(S_n)_{n\geq 0}$ on a countable group $G$, let
\[\gamma=P(S_n\neq e,\forall n\ge1)>0.\]
Suppose that $f:\mathbb{Z}_+\to\mathbb{R}$ with $f(0)=0$ satisfy
\begin{equation}\label{eq: assumption for cvg a.s.}
 \sum_{j=1}^\infty|f(j)|(1-\gamma)^j<\infty.
\end{equation}
Then, as $n\to\infty$,
\[\begin{aligned}\frac{1}{n}G_n(f)\to\gamma^2\sum_{j=1}^\infty f(j)(1-\gamma)^{j-1}\end{aligned}\]
in $L^{1}$ and with probability $1$.
\end{theorem}

\begin{remark}
The condition of the function $f$ in \eqref{eq: assumption for cvg a.s.} is weaker than the condition of the result of Asymont and Korshunov \cite[Theorem~1]{AsymontKorshunovMR4166201}. Moreover, to get a almost surely convergence or $L^{1}$ convergence for general random walks, it is plausible that our condition is optimal. For the proof techniques, our result is based on the subadditive ergodic theorem, while Asymont and Korshunov proved $L^2$ convergence and almost sure convergence in \cite[Theorem~1]{AsymontKorshunovMR4166201} via estimates on the mean and the variance of $G_n(f)$.
\end{remark}

\begin{remark}
In \cite{AsymontKorshunovMR4166201}, Asymont and Korshunov also mentioned the condition \eqref{eq: assumption for cvg a.s.}. They have found some evidence to explain that the condition \eqref{eq: assumption for cvg a.s.} is reasonable and necessary. Moreover, they prove the almost sure convergence under some additional technical assumptions in \cite[Theorem~4]{AsymontKorshunovMR4166201}.
\end{remark}

We have also considered $L^2$ convergence.
\begin{theorem}\label{thm: l^2}
For a transient random walk $(S_n)_{n\geq 0}$ on a countable group $G$, let
\[\gamma=P(S_n\neq e,\forall n\ge1)>0.\]
Suppose that $f:\mathbb{Z}_+\to\mathbb{R}$ with $f(0)=0$ satisfy
\begin{equation}\label{eq: assumption for cvg in mean square}
 \sum_{j=1}^\infty{f(j)}^2(1-\gamma)^j/j<\infty.
\end{equation}
Then, as $n\to\infty$,
\[\begin{aligned}\frac{1}{n}G_n(f)\to\gamma^2\sum_{j=1}^\infty f(j)(1-\gamma)^{j-1}\end{aligned}\]
in $L^2$.
\end{theorem}
\begin{remark}\label{rem: quasi optimal}
Note that the condition \eqref{eq: assumption for cvg in mean square} is weaker than the condition in \cite[Theorem~1]{AsymontKorshunovMR4166201}. Moreover, our condition is quasi optimal for $L^2$ convergence of $G_n(f)/n$. Indeed, consider a transient random walk such that $E(\tau|\tau<\infty)<\infty$, where $\tau=\inf\{n\geq 1:S_n=e\}$. (In particular, simple random walks in $\mathbb{Z}^d$ with $d\geq 5$ satisfy this condition.) Take $f(j)=(1-\gamma)^{-j/2}$. Then the condition \eqref{eq: assumption for cvg in mean square} is not fulfilled. But for any $\delta>0$, we have $\sum_{j=1}^{\infty}f(j)^2(1-\gamma)^{j}/j^{1+\delta}<\infty$. We will  show in Section~\ref{sec: proof remark} that $G_n(f)/n$ does not converge in $L^2$ as $n\to\infty$.
\end{remark}

\section{Preliminaries}
In this section, we state Liggett's version \cite{LiggettMR806224} of the subadditive ergodic theorem.

\begin{lemma}[Subadditive Ergodic Theorem]\label{lem: SET}
Let $(X_{m,n})_{0\leq m<n}$ be a collection of random variables. Assuming that $(X_{m,n})_{0\leq m<n}$ satisfies the following four assumptions:
\begin{enumerate}
    \item $X_{0,n} \leq X_{0,m}+X_{m,n}$, whenever $0<m<n$.
    \item The joint distribution of $\{ X_{m+1,m+k+1}, k\geq 1\} $ is the same as that of $\{X_{m,m+k}, k\geq 1\} $ for each $m\geq 0$.
    \item For each $k\geq 1$, $\{ X_{nk,(n+1)k}, n\geq 1\} $ is a stationary process.
    \item For each $n$, $E|X_{0,n}|<\infty$ and $EX_{0,n}\geq -cn$ for some constant $c$.
\end{enumerate}
Then, we have that
\[\alpha=\lim_{n\to\infty}\frac{1}{n}EX_{0,n}=\inf\{EX_{0,n}/n,n\geq 1\},\]
\[X=\lim_{n\to\infty}\frac{X_{0,n}}{n}\text{ exists a.s. and in }L^{1},\]
\[EX=\alpha.\]
If the stationary process $\{ X_{nk,(n+1)k}, n\geq 1\}$ is ergodic for each $k\geq 1$, then $X=\alpha\text{ a.s.}$.
\end{lemma}

\section{Proof of Theorem~\ref{thm: main}}
Firstly, we give a lemma.
\begin{lemma}\label{lem: Gmnf via increments}
Let $f$ be the function in Theorem~\ref{thm: main}. Then for all $0\leq m<n$, there exists a measurable function $f_{n-m}$ such that
\[G_{m,n}(f)=f_{n-m}(\xi_{m+1},\xi_{m+2},\ldots,\xi_{n-1}).\]
\end{lemma}
\begin{proof}
\begin{align*}
    G_{m,n}(f)&=\sum_{x\in G}f(\ell(n-1,x)-\ell(m-1,x))\\
    &=\sum_{x\in G}f\left(\sum_{i=m}^{n-1}\mathbb{I}_{(S_i=x)}\right)\\
    &=\sum_{x\in G}f\left(\mathbb{I}_{(S_{m}=x)}+\sum_{i=m+1}^{n-1}\mathbb{I}_{(S_{m}\xi_{m+1}\xi_{m+2}\cdots\xi_{i}=x)}\right)\\
    &\overset{y=S_{m}^{-1}x}{=}\sum_{y\in G}f\left(\mathbb{I}_{(y=e)}+\sum_{i=m+1}^{n-1}\mathbb{I}_{(\xi_{m+1}\xi_{m+2}\cdots\xi_{i}=y)}\right)\\
    &:=f_{n-m}(\xi_{m+1},\xi_{m+2},\ldots,\xi_{n-1}).
\end{align*}
\end{proof}

In the following, we  prove  Theorem~\ref{thm: main}.

\begin{proof}[Proof of Theorem~\ref{thm: main}]
Without loss of generality, we assume that $f$ is a nonnegative function.

Define the range of $k$-multiple points
\[R_n^{(k)}=\sum_{x\in G}\mathbb{I}_{(\ell(n-1,x)=k)},\]
i.e. the number of points that are visited exactly $k$ times up to time $n-1$. Thus $R^{(k)}_n=0$ for all $k>n$ and
\begin{align*}
    G_n(f)&=\sum_{x\in G}f(\ell(n-1,x))\left[\sum^{\infty}_{k=1}\mathbb{I}_{(\ell(n-1,x)=k)}\right]\\
    &=\sum^{\infty}_{k=1}f(k)\left[\sum_{x\in G}\mathbb{I}_{(\ell(n-1,x)=k)}\right]\\
    &=\sum^{\infty}_{k=1}f(k)R_n^{(k)}.
\end{align*}

For $j\geq 1$, define
\begin{equation}\label{eq: defn h j}
h^{(j)}(\ell):=\max(\ell+1-j,0).
\end{equation}
Then, we have that
\begin{equation}\label{eq: Gn hj sum}
  G_n(h^{(j)})=\sum_{k=j}^{n}(k+1-j)R_n^{(k)}.
\end{equation}
Let $R_n=\sum_{k=1}^\infty R_n^{(k)}$ be the range of random walk up to time $n-1$. Note that $n=\sum_{k=1}^{n}kR_n^{(k)}$. Hence, we have
\begin{equation}\label{eq: Gnhj with range and j range}
   n-G_n(h^{(j)})=\sum_{k=1}^{j-1}kR_n^{(k)}+(j-1)\sum_{k=j}^{n}R_n^{(k)}=(j-1)R_n-\sum_{k=1}^{j-1}(j-1-k)R_n^{(k)}.
\end{equation}
By \cite{PittMR386021, DerriennicMR588163}, we have that
\begin{equation}\label{eq: Rn large number law}
    \lim_{n\to\infty}R_n/n=\gamma,\quad \lim_{n\to\infty}R_n^{(k)}/n=\gamma^2(1-\gamma)^{k-1}, k\geq 1
\end{equation}
where the convergence is in $L^{1}$ and with probability $1$. By \eqref{eq: Gnhj with range and j range} and \eqref{eq: Rn large number law} we have that
\begin{align}\label{eq: G_nh/n L and a.s convergence}
    \lim_{n\to\infty}\frac{G_n(h^{(j)})}{n}&=1-(j-1)\lim_{n\to\infty}\frac{R_n}{n}+\sum^{j-1}_{k=1}(j-1-k)\lim_{n\to\infty}\frac{R^{(k)}_n}{n}\notag\\
    &=1-(j-1)\gamma+\sum^{j-1}_{k=1}(j-1-k)\gamma^2(1-\gamma)^{k-1}\notag\\
    &=(1-\gamma)^{j-1}
\end{align}
in $L^1$ and with probability $1$.

For an arbitrary fixed $x\in G$, $0\leq m<n$, it holds that
\[h^{(j)}(\ell(m-1,x))+h^{(j)}(\ell(n-1,x)-\ell(m-1,x))\leq h^{(j)}(\ell(n-1,x)).\]
Hence, we get that
\begin{equation}\label{eq: G superadditive}
 G_{0,m}(h^{(j)})+G_{m,n}(h^{(j)})\leq G_{0,n}(h^{(j)})
\end{equation}
and that
\begin{equation}\label{eq: EG superadditve 01}
    EG_{0,m}(h^{(j)})+EG_{m,n}(h^{(j)})\leq EG_{0,n}(h^{(j)}).
\end{equation}
By Lemma \ref{lem: Gmnf via increments}, since $(\xi_n) _{n\geq 1}$ is independent and identically distributed, we get
\begin{align*}
    EG_{m,n}(h^{(j)})&=Ef_{n-m}(\xi_{m+1}, \xi_{m+2}, \dots, \xi_{n-1})\\
    &=Ef_{n-m}(\xi_{1}, \xi_{2}, \dots, \xi_{n-m-1})\\
    &=EG_{n-m}(h^{(j)}).
\end{align*}
Hence, by \eqref{eq: EG superadditve 01} we have
\begin{equation*}
    EG_{m}(h^{(j)})+EG_{n-m}(h^{(j)})\leq EG_{n}(h^{(j)}).
\end{equation*}
Hence, by Fekete's subadditive lemma and \eqref{eq: G_nh/n L and a.s convergence} we get that
\begin{equation}\label{eq: EGhj limit}
    \sup_{n\geq 1}\frac{EG_n(h^{(j)})}{n}=\lim_{n\to\infty}\frac{EG_n(h^{(j)})}{n}=(1-\gamma)^{j-1}.
\end{equation}
By \eqref{eq: Gn hj sum}, we know that
\begin{equation}\label{eq: Rjn and Ghn}
0\leq R_n^{(j)}\leq G_n(h^{(j)}).
\end{equation}
For $p\geq1$, let $f^{(p)}(\ell)=\mathbb{I}_{[0,p]}(\ell)f(\ell)$. Then, $G_n(f^{(p)})=\sum_{j=1}^pf(j)R_n^{(j)}$. We have that
\begin{equation}\label{eq: tail control}
    0\leq G_n(f)-G_n(f^{(p)})=\sum_{j=p+1}^\infty f(j)R_n^{(j)}\leq\sum_{j=p+1}^\infty f(j)G_n(h^{(j)}).
\end{equation}
Therefore, by \eqref{eq: EGhj limit} and \eqref{eq: tail control} we have that
\begin{equation}\label{eq: Gnf minus Gnfp}
   E|G_n(f)/n-G_n(f^{(p)})/n|\leq\frac{1}{n}\sum_{j=p+1}^\infty f(j)EG_n(h^{(j)})\leq\sum_{j=p+1}^\infty f(j)(1-\gamma)^{j-1}.
\end{equation}
And for an arbitrary fixed $p\geq 1$, by \eqref{eq: Rn large number law} we have that
\begin{equation}\label{eq: Gnfp convergence}
    \lim_{n\to\infty}\frac{G_n(f^{(p)})}{n}=\lim_{n\to\infty}\sum^p_{j=1}f(j)\frac{R^{(j)}_n}{n}=\gamma^2\sum^p_{j=1}f(j)(1-\gamma)^{j-1}
\end{equation}
in $L^1$ and with probability $1$. Hence, by \eqref{eq: Gnf minus Gnfp} and \eqref{eq: Gnfp convergence} we get
\begin{align*}
&\limsup_{n\to\infty}E\left|G_n(f)/n -\gamma^2\sum_{j=1}^\infty f(j)(1-\gamma)^{j-1}\right|\\
\leq & \limsup_{n\to\infty}E|G_n(f)/n-G_n(f^{(p)})/n|\\
&+\limsup_{n\to\infty}E\left|\frac{G_n(f^p)}{n}-\gamma^2\sum_{j=1}^{p} f(j)(1-\gamma)^{j-1}\right|\\
&+\gamma^2\sum_{j=p+1}^{\infty}f(j)(1-\gamma)^{j-1}\\
\leq & (\gamma^2+1)\sum_{j=p+1}^{\infty}f(j)(1-\gamma)^{j-1},
\end{align*}
which tends to $0$ as $p\to\infty$ by \eqref{eq: assumption for cvg a.s.}. Hence, we have that
\[\frac{1}{n}G_n(f)\to\gamma^2\sum_{j=1}^\infty f(j)(1-\gamma)^{j-1}, \text{as\,\,} n\to\infty,\]
where the convergence is in $L^{1}$.

Next, we use the subadditive ergodic theorem to demonstrate that the convergence also exists almost surely. For fixed $p\geq 1$, let $Y_{m,n}=-\sum_{j=p+1}^{\infty}f(j)G_{m,n}(h^{(j)})$, $0\leq m<n$. In the following, we verify that $Y_{m,n}, 0\leq m<n$ satisfies the four assumptions of the subadditive ergodic theorem (see Lemma~\ref{lem: SET}).

Firstly, by the superadditivity of $G_n(h^{(j)})$ in \eqref{eq: G superadditive} and positivity of $f(j)$, we have that for all $0<m<n$,
\[Y_{0,n}\leq Y_{0,m}+Y_{m,n}.\]

Secondly, by the definition of $h^{(j)}$ and $G_{m,n}(h^{(j)})$, we get that for all $j\geq n-m+1$, $G_{m,n}(h^{(j)})=0$. Thus, by  Lemma~\ref{lem: Gmnf via increments}, for $0\leq m<n$, there exsits a function $h_{n-m}$ such that
\begin{align}\label{eq: Ymn and hn-m}
    Y_{m,n}&=-\sum_{j=p+1}^{n-m}f(j)G_{m,n}(h^{(j)})\notag\\
    &=-\sum_{j=p+1}^{n-m}f(j)h^{(j)}_{n-m}(\xi_{m+1}, \xi_{m+2}, \dots, \xi_{n-1})\notag\\
    &:=h_{n-m}(\xi_{m+1}, \xi_{m+2}, \dots, \xi_{n-1}).
\end{align}
Hence, for all $m\geq 0$ and $k\geq 1$,
\[Y_{m+1,m+k+1}=h_{k}(\xi_{m+2}, \xi_{m+3}, \dots, \xi_{m+k}), \,\,Y_{m,m+k}=h_{k}(\xi_{m+1}, \xi_{m+2}, \dots, \xi_{m+k-1}).\]
Since $\{\xi_n,n\geq 1\}$ is a sequence of independent and identically distributed random variables, the joint distribution of $\{Y_{m+1,m+k+1},k\geq 1\}$ is the same as $\{Y_{m,m+k},k\geq 1\}$.

Next, by \eqref{eq: Ymn and hn-m} we have that
\[Y_{nk,(n+1)k}=h_k(\xi_{nk+1}, \xi_{nk+2}, \dots, \xi_{(n+1)k-1}).\]
Since $\{\xi_n,n\geq 1\}$ is a sequence of independent and identically distributed random variables, for each fixed $k\geq 1$,  $\{Y_{nk, (n+1)k},n\geq 1\}$ is also a sequence of independent and identically distributed random variables, from which we deduce the stationarity and ergodicity of $\{Y_{nk, (n+1)k}, n\geq 1\}$.

Thirdly, by \eqref{eq: EGhj limit} we have that for any $n\geq 1$,
\begin{align*}
E|Y_{0,n}|&=E\left[\sum^{\infty}_{j=p+1}f(j)G_n(h^{(j)})\right]\\
&=\sum^{\infty}_{j=p+1}f(j)EG_n(h^{(j)})\\
&\leq n\sum^{\infty}_{j=p+1}f(j)(1-\gamma)^{j-1}<\infty.
\end{align*}
And let $c=\sum^{\infty}_{j=1}f(j)(1-\gamma)^{j-1}$, we also get
\[EY_{0,n}=-E|Y_{0,n}|\geq -cn.\]

Hence, by the subadditive ergodic theorem (see Lemma~\ref{lem: SET}) there exists a constant $\beta_p\geq 0$ such that
\[-\beta_p=\lim_{n\to\infty}EY_{0,n}/n=\inf\{EY_{0,n}/n, n\geq 1\},\]
and
\[\lim_{n\to\infty}\frac{Y_{0,n}}{n}=-\beta_p\]
in $L^1$ and with probability $1$. Equivalently,
\begin{equation}\label{eq: Gnh and beta p}
    \lim_{n\to\infty}\frac{1}{n}\sum_{j=p+1}^{\infty}f(j)G_{n}(h^{(j)})=\beta_p
\end{equation}
in $L^1$ and with probability $1$. Then by Fatou's lemma and \eqref{eq: Gnf minus Gnfp}, we can get that
\begin{align}\label{eq: upper bound of beta p}
    \beta_p&=E\left(\lim_{n\to\infty}\frac{1}{n}\sum_{j=p+1}^{\infty}f(j)G_{n}(h^{(j)})\right)\notag\\
    &\leq \liminf_{n\to\infty}E\left(\frac{1}{n}\sum_{j=p+1}^{\infty}f(j)G_{n}(h^{(j)})\right)\notag\\
    &=\liminf_{n\to\infty}\frac{1}{n}\sum_{j=p+1}^{\infty}f(j)EG_{n}(h^{(j)})\notag\\
    &\leq \sum^{\infty}_{j=p+1}f(j)(1-\gamma)^{j-1}.
\end{align}
Hence, by \eqref{eq: tail control}, \eqref{eq: Gnfp convergence}, \eqref{eq: Gnh and beta p} and \eqref{eq: upper bound of beta p} we get that
\begin{align*}
    &\limsup_{n\to\infty}\left|\frac{G_n(f)}{n}-\gamma^2\sum^{\infty}_{j=1}f(j)(1-\gamma)^{j-1}\right|\\
    &\leq \limsup_{n\to\infty}\left|\frac{G_n(f)}{n}-\frac{G_n(f^{(p)})}{n}\right|+\limsup_{n\to\infty}\left|\frac{G_n(f^{(p)})}{n}-\gamma^2\sum^p_{j=1}f(j)(1-\gamma)^{j-1}\right|\\
    &\quad +\gamma^2\sum^{\infty}_{j=p+1}f(j)(1-\gamma)^{j-1}\\
    &\leq \limsup_{n\to\infty}\frac{1}{n}\sum^{\infty}_{j=p+1}f(j)G_n(h^{(j)})+\gamma^2\sum^{\infty}_{j=p+1}f(j)(1-\gamma)^{j-1}\\
    &=\beta_p+\gamma^2\sum^{\infty}_{j=p+1}f(j)(1-\gamma)^{j-1}\\
    &\leq (1+\gamma^2)\sum^{\infty}_{j=p+1}f(j)(1-\gamma)^{j-1},
\end{align*}
which tends to $0$ as $p\to\infty$. Therefore, as $n\to\infty$, $G_n(f)/n$ tends to $\gamma^2\sum^{\infty}_{j=1}f(j)(1-\gamma)^{j-1}$ with probability $1$.
\end{proof}

\section{Proof of Theorem~\ref{thm: l^2}}

Firstly, by \eqref{eq: assumption for cvg in mean square} and Cauchy-Schwarz inequality we have
\[\sum^{\infty}_{j=1}|f(j)|(1-\gamma)^j\leq \left(\sum^{\infty}_{j=1}\frac{f^2(j)}{j}(1-\gamma)^{j}\right)^{\frac{1}{2}}\left(\sum^{\infty}_{j=1}j(1-\gamma)^j\right)^{\frac{1}{2}}<\infty,\]
which means that \eqref{eq: assumption for cvg a.s.} holds. Hence, Theorem~\ref{thm: main} is applicable.

Next, without loss of generality, we still assume that $f$ is a nonnegative function.  For $p\geq 1$, define $f^{(p)}(\ell)=\mathbb{I}_{[0,p]}(\ell)f(\ell)$ as the same as in the proof of Theorem ~\ref{thm: main}. Since $0\leq\gamma\leq 1$ and $(a+b+c)^2\leq 3(a^2+b^2+c^2)$, we have that
\begin{align}\label{eq: L2 convergence main}
&E\left[\left(\frac{G_{n}(f)}{n}-\gamma^{2}\sum_{j=1}^{\infty}f(j)(1-\gamma)^{j-1}\right)^2\right]\notag\\
\leq & 3E\left[\left(\frac{G_{n}(f)}{n}-\frac{G_{n}(f^{(p)})}{n}\right)^2\right]\notag\\
 &+3E\left[\left(\frac{G_n(f^{(p)})}{n}-\gamma^2\sum^p_{j=1}f(j)(1-\gamma)^{j-1}\right)^2\right]\notag\\
 &+3\left(\sum^{\infty}_{j=p+1}f(j)(1-\gamma)^{j-1}\right)^2.
\end{align}

Next, by Cauchy-Schwarz inequality, the equality $n=\sum_{j=1}^{n}jR_n^{(j)}=\sum^{\infty}_{j=1}jR_n^{(j)}$, \eqref{eq: Rjn and Ghn} and \eqref{eq: tail control}, we can get
\begin{align*}
 \left(\frac{G_{n}(f)}{n}-\frac{G_{n}(f^{(p)})}{n}\right)^2&=\frac{1}{n^2}\left(\sum^{\infty}_{j=p+1}f(j)R^{(j)}_n\right)^2\\
 &\leq \frac{1}{n^2}\left(\sum^{\infty}_{j=p+1}\frac{f^2(j)}{j}R^{(j)}_n\right)\left(\sum^{\infty}_{j=p+1}jR^{(j)}_n\right)\\
 &\leq \frac{1}{n}\left(\sum^{\infty}_{j=p+1}\frac{f^2(j)}{j}R^{(j)}_n\right)\\
 &\leq \sum^{\infty}_{j=p+1}\frac{f^2(j)}{j}\frac{G_n(h^{(j)})}{n}.
\end{align*}
Therefore, by \eqref{eq: EGhj limit} we have that for all $n\geq 1$,
\begin{equation}\label{eq: Gnf-Gnfp L2 inequ}
E\left[\left(\frac{G_{n}(f)}{n}-\frac{G_{n}(f^{(p)})}{n}\right)^2\right] \leq \sum^{\infty}_{j=p+1}\frac{f^2(j)}{j}\cdot\frac{EG_n(h^{(j)})}{n}\leq \sum^{\infty}_{j=p+1}\frac{f^2(j)}{j}(1-\gamma)^{j-1}.
\end{equation}

Next, for any fixed $p\geq 1$, the random variable $\frac{G_n(f^{(p)})}{n}$ is bounded as
\[\left|\frac{G_n(f^{(p)})}{n}\right|=\sum^p_{j=1}f(j)\frac{R^{(j)}_n}{n}\leq \sum^p_{j=1}f(j)\frac{n}{n}\leq p\max_{1\leq j\leq p}f(j).\]
Hence, by \eqref{eq: Gnfp convergence} and the dominated convergence theorem, we have that
\begin{equation}\label{eq: Gnfp L2 convergence}
    \lim_{n\to\infty}E\left[\left(\frac{G_n(f^{(p)})}{n}-\gamma^2\sum^p_{j=1}f(j)(1-\gamma)^{j-1}\right)^2\right]=0.
\end{equation}

Next, by Cauchy-Schwarz inequality again, it holds that
\begin{align}\label{eq: sum f L2}
    \left(\sum^{\infty}_{j=p+1}f(j)(1-\gamma)^{j-1}\right)^2
    \leq& \left(\sum^{\infty}_{j=p+1}\frac{f^2(j)}{j}(1-\gamma)^{j-1}\right)\left(\sum^{\infty}_{j=p+1}j(1-\gamma)^{j-1}\right)\notag\\
    \leq& C\left(\sum^{\infty}_{j=p+1}\frac{f^2(j)}{j}(1-\gamma)^{j-1}\right),
\end{align}
where $C=\sum^{\infty}_{j=1}j(1-\gamma)^{j-1}$ is a positive constant independent of $p,n$ and $f$.

Finally, by \eqref{eq: L2 convergence main}, \eqref{eq: Gnf-Gnfp L2 inequ}, \eqref{eq: Gnfp L2 convergence} and  \eqref{eq: sum f L2}, we get
\begin{align*}
    &\limsup_{n\to\infty}E\left[\left(\frac{G_{n}(f)}{n}-\gamma^{2}\sum_{j=1}^{\infty}f(j)(1-\gamma)^{j-1}\right)^2\right]\\
    \leq& 3\limsup_{n\to\infty}E\left[\left(\frac{G_{n}(f)}{n}-\frac{G_{n}(f^{(p)})}{n}\right)^2\right]\\
    &+3\limsup_{n\to\infty}E\left[\left(\frac{G_n(f^{(p)})}{n}-\gamma^2\sum^p_{j=1}f(j)(1-\gamma)^{j-1}\right)^2\right]\\
    &+3\left(\sum^{\infty}_{j=p+1}f(j)(1-\gamma)^{j-1}\right)^2\\
  \leq& 3(C+1)\sum^{\infty}_{j=p+1}\frac{f^2(j)}{j}(1-\gamma)^{j-1},
\end{align*}
which tends to $0$ as $p\to\infty$ by \eqref{eq: assumption for cvg in mean square}. Hence, we get
\[\limsup_{n\to\infty}E\left[\left(\frac{G_{n}(f)}{n}-\gamma^{2}\sum_{j=1}^{\infty}f(j)(1-\gamma)^{j-1}\right)^2\right]=0,\]
i.e. as $n\to\infty$, $G_n(f)/n$ converges to $\gamma^2\sum_{j=1}^\infty f(j)(1-\gamma)^{j-1}$ in $L^2$.

\section{Proof of Remark~\ref{rem: quasi optimal}}\label{sec: proof remark}
We will prove Remark~\ref{rem: quasi optimal} in this section. Firstly, we give a lower bound of $ER_{n}^{(j)}$ as follows:
\begin{lemma}\label{lem: lower bound of the expectation of R_{n}^{(j)}}
Define $\tau_0=0$. For $j\geq 1$, we recursively define
\[\tau_{j}=\inf\{n>\tau_{j-1}:S_n=e\}\]
with the convention that $\inf\emptyset=\infty$. In other words, $\tau_{j}$ is the $j$-th return time to the initial position $e$. Then, we have the following inequality:
\[ER_{n}^{(j)}\geq \gamma^{2}E\max(n-\tau_{j-1},0).\]
\end{lemma}
\begin{proof}
Since
\[R_{n}^{(j)}=\sharp\left\{k\in[0,n-1]: E_{k}F_{k,n}^{(j)}\text{ occurs}\right\},\]
where
\[E_{k}=\left\{S_{m}\neq S_{k}, \forall 0\leq m<k\right\}, F_{k,n}^{(j)}=\left\{\sum_{m=k}^{n-1}\mathbb{I}_{(S_{m}=S_{k})}=j\right\},\]
we have
\[ER_{n}^{(j)}=\sum_{k=0}^{n-1}P(E_{k}F_{k,n}^{(j)}).\]
By independence between $(S_{m})_{0\leq m\leq k}$ and $(S^{-1}_{k}S_{k+m})_{0\leq m\leq n-1-k}$, we know that $E_{k}$
and $F_{k,n}^{(j)}$ are independent. Hence,
\[ER_{n}^{(j)}=\sum_{k=0}^{n-1}P(E_{k})P(F_{k,n}^{(j)}).\]
Moreover, we have that
\[P(E_{k})=P(S_m\neq e, \forall 1\leq m\leq k)\geq P(S_m\neq e, \forall m\geq1)=\gamma\]
and
\begin{align*}
P(F_{k,n}^{(j)})&=P\left(\sum^{n-1}_{m=k}\mathbb{I}_{(S^{-1}_kS_m=e)}=j\right)\\
&=P\left(\sum_{m=0}^{n-1-k}\mathbb{I}_{(S_m=e)}=j\right)\\
&=\sum_{\ell=0}^{n-1-k}P(\tau_{j-1}=\ell\text{ and }\forall\ell<m\leq n-1-k,S_m\neq e)\\
&=\sum_{\ell=0}^{n-1-k}P(\tau_{j-1}=\ell)P(\forall 0<m\leq n-1-k-\ell,S_m\neq e)\\
&\geq\sum_{\ell=0}^{n-1-k}P(\tau_{j-1}=\ell)P(\forall m\geq 1, S_m\neq e)\\
&=\gamma P\left(\tau_{j-1}\leq n-1-k\right).
\end{align*}
Hence, we have that
\begin{equation*}
ER_{n}^{(j)}\geq\gamma^2\sum_{k=0}^{n-1}P(\tau_{j-1}\leq n-1-k)=\gamma^2E\max(n-\tau_{j-1},0).
\end{equation*}
\end{proof}

\begin{proof}[Proof of Remark~\ref{rem: quasi optimal}]

Since $f(j)=(1-\gamma)^{-\frac{j}{2}}$, we have
\[\sum^{\infty}_{j=1}(f(j))^2(1-\gamma)^j/j=\sum^{\infty}_{j=1}\frac{1}{j}=\infty,\]
which means the condition \eqref{eq: assumption for cvg in mean square} is not fulfilled. And for any $\delta>0$, we have
\[\sum_{j=1}^{\infty}(f(j))^2(1-\gamma)^{j}/j^{1+\delta}=\sum^{\infty}_{j=1}\frac{1}{j^{1+\delta}}<\infty.\]

Let $(S_n)_{n\geq 0}$ be a transient random walk such that $E(\tau|\tau<\infty)<\infty$, where $\tau=\inf\{n\geq 1:S_n=e\}$.  We assume that $G_n(f)/n$ converges in $L^2$ as $n\to\infty$.

Fix $0<\delta<1$. By Cauchy-Schwarz inequality,
\begin{equation*}\label{eq: f(j) condition 1 remark}
\sum^{\infty}_{j=1}f(j)(1-\gamma)^{j-1}\leq \left(\sum^{\infty}_{j=1}\frac{(f(j))^2}{j^{1+\delta}}(1-\gamma)^{j-1}\right)^{\frac{1}{2}}\left(\sum^{\infty}_{j=1}j^{1+\delta}(1-\gamma)^{j-1}\right)^{\frac{1}{2}}<\infty.
\end{equation*}
Hence, by Theorem~\ref{thm: main} we get the almost sure convergence of $\frac{G_n(f)}{n}$ to $\gamma^2\sum^{\infty}_{j=1}f(j)(1-\gamma)^{j-1}$ as $n\to\infty$. As we have already assumed that $G_n(f)/n$ converges in $L^2$ as $n\to\infty$, we get that $\frac{G_n(f)}{n}$ also converges in $L^2$ to $\gamma^2\sum^{\infty}_{j=1}f(j)(1-\gamma)^{j-1}$ as $n\to\infty$.

For $p\geq 1$, let $f^{(p)}(\ell)=\mathbb{I}_{[0,p]}(\ell)f(\ell)$ as the same as in the proof of Theorem ~\ref{thm: main}.  Then
\[\sum^{\infty}_{j=1}\frac{(f^{(p)}(j))^2}{j}(1-\gamma)^{j-1}=\sum^p_{j=1}\frac{(f(j))^2}{j}(1-\gamma)^{j-1}<\infty.\]
By Theorem~\ref{thm: l^2}, $\frac{G_n(f^{(p)})}{n}$ converges in $L^2$ to $\gamma^2\sum^{p}_{j=1}f(j)(1-\gamma)^{j-1}$ as $n\to\infty$. Hence, $\frac{G_n(f)}{n}-\frac{G_n(f^{(p)})}{n}$ also converges in $L^2$ and
\begin{equation}\label{eq: Gf-Gfp remark}
\lim_{n\to\infty}E\left[\left(\frac{G_n(f)}{n}-\frac{G_n(f^{(p)})}{n}\right)^2\right]=\left(\gamma^2\sum^{\infty}_{j=p+1}f(j)(1-\gamma)^{j-1}\right)^2.
\end{equation}

On the other hand,
\begin{align*}
	\left(\frac{G_n(f)}{n}-\frac{G_n(f^{(p)})}{n}\right)^2&=\frac{1}{n^2}\left(\sum^{\infty}_{j=p+1}f(j)R^{(j)}_n\right)^2\\
	&=\frac{1}{n^2}\sum_{j_1,j_2\geq p+1}f(j_1)f(j_2)R^{(j_1)}_nR^{(j_2)}_n\\
	&\geq \frac{1}{n^2}\sum^{\infty}_{j=p+1}(f(j))^2(R^{(j)}_n)^2\\
	&\geq \frac{1}{n^2}\sum^{\infty}_{j=p+1}(f(j))^2R^{(j)}_n.
\end{align*}
Hence,
\begin{equation}\label{eq: Gf-Gfp L2 remark}
 E\left[\left(\frac{G_n(f)}{n}-\frac{G_n(f^{(p)})}{n}\right)^2\right]\geq \frac{1}{n^2}\sum^{\infty}_{j=1}(f(j))^2ER^{(j)}_n.
\end{equation}
Let $a=E(\tau|\tau<\infty)<\infty$. Then, we have that
\[E(\tau_{j-1}|\tau_{j-1}<\infty)=a(j-1).\]
Then, for all $j\leq 1+\frac{n}{2a}$, we have that $E(n-\tau_{j-1}|\tau_{j-1}<\infty)=n-a(j-1)\geq n/2$ and that
\begin{equation}\label{eq: tau max}
E\max(n-\tau_{j-1},0)\geq P(\tau_{j-1}<\infty)E(n-\tau_{j-1}|\tau_{j-1}<\infty)\geq \frac{n}{2}(1-\gamma)^{j-1}.
\end{equation}
Recall that $f(j)=(1-\gamma)^{-\frac{j}{2}}$. By Lemma~\ref{lem: lower bound of the expectation of R_{n}^{(j)}}, \eqref{eq: Gf-Gfp L2 remark} and \eqref{eq: tau max}, we have that
\begin{align*}
E\left[\left(\frac{G_n(f)}{n}-\frac{G_n(f^{(p)})}{n}\right)^2\right]&\geq \frac{\gamma^2}{n^2}\sum^{\infty}_{j=p+1}(f(j))^2E\max(n-\tau_{j-1},0)\\
&\geq \frac{\gamma^2}{n^2}\sum^{1+\lfloor \frac{n}{2a}\rfloor}_{j=p+1}(f(j))^2E\max(n-\tau_{j-1},0)\\
&\geq \frac{\gamma^2}{n^2}\sum^{1+\lfloor \frac{n}{2a}\rfloor}_{j=p+1}(1-\gamma)^{-j}\cdot \frac{n}{2}(1-\gamma)^{j-1}\\
&\geq \frac{\gamma^2}{2n(1-\gamma)}\left(\frac{n}{2a}-p\right).
\end{align*}
Hence, we have that
\begin{equation*}
\liminf_{n\to\infty}E\left[\left(\frac{G_n(f)}{n}-\frac{G_n(f^{(p)})}{n}\right)^2\right]\geq \frac{\gamma^2}{4a(1-\gamma)},
\end{equation*}
which contradicts with \eqref{eq: Gf-Gfp remark} for large enough $p$. Hence, $G_n(f)/n$ does not converge in $L^2$ as $n\to\infty$.
\end{proof}
\bibliographystyle{alpha}
\bibliography{group}

@article {AsymontKorshunovMR4166201,
    AUTHOR = {Asymont, Inna M. and Korshunov, Dmitry},
     TITLE = {Strong law of large numbers for a function of the local times
              of a transient random walk in {${\Bbb {Z}}^d$}},
   JOURNAL = {J. Theoret. Probab.},
  FJOURNAL = {Journal of Theoretical Probability},
    VOLUME = {33},
      YEAR = {2020},
    NUMBER = {4},
     PAGES = {2315--2336},
      ISSN = {0894-9840,1572-9230},
   MRCLASS = {60G50 (60F15 60J55)},
  MRNUMBER = {4166201},
MRREVIEWER = {Endre\ Cs\'aki},
       DOI = {10.1007/s10959-019-00937-6},
       URL = {https://doi.org/10.1007/s10959-019-00937-6},
}

@article {LiggettMR806224,
    AUTHOR = {Liggett, Thomas M.},
     TITLE = {An improved subadditive ergodic theorem},
   JOURNAL = {Ann. Probab.},
  FJOURNAL = {The Annals of Probability},
    VOLUME = {13},
      YEAR = {1985},
    NUMBER = {4},
     PAGES = {1279--1285},
      ISSN = {0091-1798,2168-894X},
   MRCLASS = {60F15 (60K35)},
  MRNUMBER = {806224},
MRREVIEWER = {Yves\ Derriennic},
       DOI = {10.1214/aop/1176992811},
       URL = {https://doi.org/10.1214/aop/1176992811}
}

@article {PittMR386021,
    AUTHOR = {Pitt, Joel H.},
     TITLE = {Multiple points of transient random walks},
   JOURNAL = {Proc. Amer. Math. Soc.},
  FJOURNAL = {Proceedings of the American Mathematical Society},
    VOLUME = {43},
      YEAR = {1974},
     PAGES = {195--199},
      ISSN = {0002-9939,1088-6826},
   MRCLASS = {60J15},
  MRNUMBER = {386021},
MRREVIEWER = {Leopold\ Flatto},
       DOI = {10.2307/2039355},
       URL = {https://doi.org/10.2307/2039355},
}

@inproceedings {DvoretzkyErdosMR47272,
    AUTHOR = {Dvoretzky, A. and Erd\"os, P.},
     TITLE = {Some problems on random walk in space},
 BOOKTITLE = {Proceedings of the {S}econd {B}erkeley {S}ymposium on
              {M}athematical {S}tatistics and {P}robability, 1950},
     PAGES = {353--367},
 PUBLISHER = {Univ. California Press, Berkeley-Los Angeles, Calif.},
      YEAR = {1951},
   MRCLASS = {60.0X},
  MRNUMBER = {47272},
MRREVIEWER = {S.\ Kakutani},
}

@article {BeckerKonigMR2501325,
    AUTHOR = {Becker, Mathias and K\"onig, Wolfgang},
     TITLE = {Moments and distribution of the local times of a transient
              random walk on {$\Bbb Z^d$}},
   JOURNAL = {J. Theoret. Probab.},
  FJOURNAL = {Journal of Theoretical Probability},
    VOLUME = {22},
      YEAR = {2009},
    NUMBER = {2},
     PAGES = {365--374},
      ISSN = {0894-9840,1572-9230},
   MRCLASS = {60G50 (60F15 60J55)},
  MRNUMBER = {2501325},
       DOI = {10.1007/s10959-008-0168-4},
       URL = {https://doi.org/10.1007/s10959-008-0168-4},
}

@book {SpitzerMR388547,
    AUTHOR = {Spitzer, Frank},
     TITLE = {Principles of random walk},
    SERIES = {Graduate Texts in Mathematics},
    VOLUME = {Vol. 34},
   EDITION = {Second},
 PUBLISHER = {Springer-Verlag, New York-Heidelberg},
      YEAR = {1976},
     PAGES = {xiii+408},
   MRCLASS = {60J15},
  MRNUMBER = {388547},
}

@incollection {DerriennicMR588163,
    AUTHOR = {Derriennic, Yves},
     TITLE = {Quelques applications du th\'eor\`eme ergodique sous-additif},
 BOOKTITLE = {Conference on {R}andom {W}alks ({K}leebach, 1979) ({F}rench)},
    SERIES = {Ast\'erisque},
    VOLUME = {74},
     PAGES = {183--201, 4},
 PUBLISHER = {Soc. Math. France, Paris},
      YEAR = {1980},
   MRCLASS = {60B15 (28D20 60J15)},
  MRNUMBER = {588163},
MRREVIEWER = {Philippe\ Bougerol},
}

@article {OkamuraMR4289898,
    AUTHOR = {Okamura, Kazuki},
     TITLE = {Some results for range of random walk on graph with spectral
              dimension two},
   JOURNAL = {J. Theoret. Probab.},
  FJOURNAL = {Journal of Theoretical Probability},
    VOLUME = {34},
      YEAR = {2021},
    NUMBER = {3},
     PAGES = {1653--1688},
      ISSN = {0894-9840,1572-9230},
   MRCLASS = {60K35 (60G50)},
  MRNUMBER = {4289898},
MRREVIEWER = {Longmin\ Wang},
       DOI = {10.1007/s10959-020-01013-0},
       URL = {https://doi.org/10.1007/s10959-020-01013-0},
}

@article {MrazovicSandricSebekMR4643004,
    AUTHOR = {Mrazovi\'c, Rudi and Sandri\'c, Nikola and \v{S}ebek, Stjepan},
     TITLE = {Capacity of the range of random walks on groups},
   JOURNAL = {Kyoto J. Math.},
  FJOURNAL = {Kyoto Journal of Mathematics},
    VOLUME = {63},
      YEAR = {2023},
    NUMBER = {4},
     PAGES = {783--805},
      ISSN = {2156-2261,2154-3321},
   MRCLASS = {60G50 (05C81 60F05)},
  MRNUMBER = {4643004},
       DOI = {10.1215/21562261-2023-0003},
       URL = {https://doi.org/10.1215/21562261-2023-0003},
}

@article {AsselahSchapiraSousiMR3945751,
    AUTHOR = {Asselah, Amine and Schapira, Bruno and Sousi, Perla},
     TITLE = {Capacity of the range of random walk on {$\Bbb{Z}^4$}},
   JOURNAL = {Ann. Probab.},
  FJOURNAL = {The Annals of Probability},
    VOLUME = {47},
      YEAR = {2019},
    NUMBER = {3},
     PAGES = {1447--1497},
      ISSN = {0091-1798,2168-894X},
   MRCLASS = {60F05 (60G50)},
  MRNUMBER = {3945751},
MRREVIEWER = {Adam\ Matthew\ Bowditch},
       DOI = {10.1214/18-AOP1288},
       URL = {https://doi.org/10.1214/18-AOP1288},
}

@article {HamanaMR1434120,
    AUTHOR = {Hamana, Yuji},
     TITLE = {The fluctuation result for the multiple point range of
              two-dimensional recurrent random walks},
   JOURNAL = {Ann. Probab.},
  FJOURNAL = {The Annals of Probability},
    VOLUME = {25},
      YEAR = {1997},
    NUMBER = {2},
     PAGES = {598--639},
      ISSN = {0091-1798,2168-894X},
   MRCLASS = {60J15 (60F05)},
  MRNUMBER = {1434120},
MRREVIEWER = {Jean\ Bertoin},
       DOI = {10.1214/aop/1024404413},
       URL = {https://doi.org/10.1214/aop/1024404413},
}

@article {HamanaMR1608981,
    AUTHOR = {Hamana, Yuji},
     TITLE = {A remark on the multiple point range of two-dimensional random
              walks},
   JOURNAL = {Kyushu J. Math.},
  FJOURNAL = {Kyushu Journal of Mathematics},
    VOLUME = {52},
      YEAR = {1998},
    NUMBER = {1},
     PAGES = {23--80},
      ISSN = {1340-6116,1883-2032},
   MRCLASS = {60J15 (60F99)},
  MRNUMBER = {1608981},
MRREVIEWER = {Endre\ Cs\'aki},
       DOI = {10.2206/kyushujm.52.23},
       URL = {https://doi.org/10.2206/kyushujm.52.23},
}

@article {DemboOkadaMR4791423,
    AUTHOR = {Dembo, Amir and Okada, Izumi},
     TITLE = {Capacity of the range of random walk: the law of the iterated
              logarithm},
   JOURNAL = {Ann. Probab.},
  FJOURNAL = {The Annals of Probability},
    VOLUME = {52},
      YEAR = {2024},
    NUMBER = {5},
     PAGES = {1954--1991},
      ISSN = {0091-1798,2168-894X},
   MRCLASS = {60G50 (60F15)},
  MRNUMBER = {4791423},
       DOI = {10.1214/24-aop1692},
       URL = {https://doi.org/10.1214/24-aop1692},
}

@article {GilchMR4718437,
    AUTHOR = {Gilch, Lorenz A.},
     TITLE = {Asymptotic capacity of the range of random walks on free
              products of graphs},
   JOURNAL = {Electron. J. Probab.},
  FJOURNAL = {Electronic Journal of Probability},
    VOLUME = {29},
      YEAR = {2024},
     PAGES = {Paper No. 33, 38},
      ISSN = {1083-6489},
   MRCLASS = {60J10 (20E06 60F05)},
  MRNUMBER = {4718437},
MRREVIEWER = {Damjan\ \v Skulj},
       DOI = {10.1214/24-ejp1086},
       URL = {https://doi.org/10.1214/24-ejp1086},
}

\end{document}